\documentclass{article}
\usepackage[margin=1.25in]{geometry}
\usepackage{mathtools,amssymb,amsthm}
\usepackage{graphicx,color}
\usepackage{subcaption}
\usepackage{pifont}
\usepackage{booktabs}
\usepackage[numbers,sort]{natbib}
\usepackage{float}
\usepackage[boxed]{algorithm2e}
\usepackage[hyphens]{url}
\usepackage{bm}

\usepackage[toc,page]{appendix}
\usepackage{hyperref}
\usepackage[capitalize]{cleveref}
\usepackage{mathtools,amsmath,amsthm, amssymb}
\usepackage{xifthen}
\usepackage{dsfont}

\newcommand*{\kl}[3][]{%
\ifthenelse{\isempty{#1}}{\operatorname{D}(#2\,\|\,#3)}%
{\operatorname{D}(#2\,\|\,#3\mid#1)}%
}

\newcommand*{\triplenorm}[1]{{\left\vert\kern-0.25ex\left\vert\kern-0.25ex\left\vert #1
    \right\vert\kern-0.25ex\right\vert\kern-0.25ex\right\vert}}

\DeclareMathOperator{\var}{Var}
\DeclareMathOperator{\cov}{Cov}

\newcommand*{\rd}{\mathrm{d}}
\newcommand*{\dd}{\, \rd}

\newcommand{\R}{\mathbb{R}}

\newcommand{\eps}{\varepsilon}

\renewcommand{\phi}{\varphi}

\newtheorem{theorem}{Theorem}

\newtheorem{proposition}{Proposition}
\newtheorem{lemma}{Lemma}
\newtheorem{remark}{Remark}

\theoremstyle{definition}
\newtheorem{assumption}{Assumption}

\DeclareMathOperator*{\minimize}{\mathrm{minimize}}
\DeclareMathOperator*{\maximize}{\mathrm{maximize}}

\usepackage{bm}
\usepackage{xcolor}
\usepackage{enumitem}

\usepackage{custom}
\usepackage[normalem]{ulem}

\hypersetup{citecolor=cyan, colorlinks=true, linkcolor=blue!75!black}

\begin{document}

\title{An entropic generalization of Caffarelli’s contraction theorem via covariance inequalities}
\author{Sinho Chewi\thanks{Department of Mathematics at Massachusetts Institute of Technology (MIT), \texttt{schewi@mit.edu}} \and Aram-Alexandre Pooladian\thanks{Center for Data Science at New York University (NYU), \texttt{aram-alexandre.pooladian@nyu.edu}}}
\date{\today}

\maketitle

\vspace*{.1in}

\begin{abstract}
The optimal transport map between the standard Gaussian measure and an $\alpha$-strongly log-concave probability measure is $\alpha^{-1/2}$-Lipschitz, as first observed in a celebrated theorem of Caffarelli.
In this paper, we apply two classical covariance inequalities (the Brascamp--Lieb and Cram\'er--Rao inequalities) to prove a sharp bound on the Lipschitz constant of the map that arises from \textit{entropically regularized} optimal transport. In the limit as the regularization tends to zero, we obtain an elegant and short proof of Caffarelli's original result.
We also extend Caffarelli's theorem to the setting in which the Hessians of the log-densities of the measures are bounded by arbitrary positive definite commuting matrices.
\end{abstract}

\section{Introduction}\label{sec:intro}

In~\cite{caffarelli2000monotonicity}, Caffarelli proved the following seminal result.
\begin{theorem}[Caffarelli's contraction theorem]\label{thm: caf_con}
Let $P = \exp(-V)$ and $Q = \exp(-W)$ have smooth densities on $\R^d$, with $\nabla^2 V \preceq \beta_V I $ and $\nabla^2 W \succeq \alpha_W I \succ 0$. Then, the optimal transport map $\nabla \phi_0$ from $P$ to $Q$ is $\sqrt{\beta_V/\alpha_W}$-Lipschitz.
\end{theorem}

Here, $\phi_0 : \R^d\to\R$ is a convex function, known as a \emph{Brenier potential}. The optimal transport map $\nabla \phi_0 : \R^d\to\R^d$ pushes forward $P$ to $Q$, in the sense that if $X$ is a random variable with law $P$, then $\nabla \phi_0(X)$ is a random variable with law $Q$. See Section~\ref{sec: background_ot} and the textbook~\cite{villani2003topics} for background on optimal transport.

Caffarelli's contraction theorem can be used to transfer functional inequalities, such as a Poincar\'e inequality, from the standard Gaussian measure on $\R^d$ to other probability measures~\cite{bakrygentilledoux2014}.
Towards this end, recent works have also constructed and studied alternative Lipschitz transport maps (e.g.~\cite{kim2012generalization,mikulincer2021brownian, mikulincer2022lipschitz, neeman2022lipschitz}), but still the properties of the original optimal transport map remain of fundamental interest, with many questions unresolved~\cite{valdimarsson2007otpotential,colombo2015lipschitz}.

Indeed, besides the application to functional inequalities, the structural properties of optimal transport maps play a fundamental role in theoretical and methodological advances in optimal transport, such as the control of the curvature of the Wasserstein space through the notion of extendible geodesics~\cite{legouicetal2019fast, ahidarlegouicparis2020barycenters}, the stability of Wasserstein barycenters~\cite{chewietal2020buresgd}, and the statistical estimation of optimal transport maps~\cite{hutter2021minimax}.

In applied domains, however, the inauspicious computational and statistical burden of solving the original optimal transport problem has instead led practitioners to consider \emph{entropically regularized} optimal transport, as pioneered in~\cite{cuturi2013sinkhorn}. In addition to its practical merits, entropic optimal transport enjoys a rich mathematical theory, rooted in its connection to the classical Schr\"odinger bridge problem~\cite{leonard2014schrodinger}, which has led to powerful applications to high-dimensional probability~\cite{ledoux2018remarks, fathi2020proof, gentiletal2020entropichwi}. As such, it is natural to study the properties of the entropic analogue of the optimal transport map.

In this paper, we prove a generalization of Caffarelli's contraction theorem to the setting of entropic optimal transport.
Namely, we study the Hessian of the \textit{entropic Brenier potential} (see Section~\ref{sec: background_eot}), which admits a representation as a covariance matrix (Lemma~\ref{lem: hessians_pi}). By applying two well-known inequalities for covariance matrices (the Brascamp--Lieb inequality and the Cram\'er--Rao inequality), we quickly deduce a sharp upper bound on the operator norm of the Hessian which holds for any value $\eps > 0$ of the regularization parameter.

As a byproduct of our analysis, by sending $\eps \searrow 0$ and appealing to recent convergence results for the entropic Brenier potentials~\cite{nutz2021entropic}, we obtain the shortest proof of Caffarelli's contraction theorem to date. Notably, our argument allows us to sidestep the regularity of the optimal transport map, which is a key obstacle in Caffarelli's original proof.

Recently, in~\cite{fathi2020proof}, Fathi, Gozlan, and Prod'homme gave a proof of Caffarelli's theorem using a surprising equivalence between Theorem~\ref{thm: caf_con} and a statement about Wasserstein projections, which was discovered through the theory of weak optimal transport~\cite{gozlanjuillet2020weakot}.
In order to verify the latter, their proof also used ideas from entropic optimal transport. In comparison, we note that our argument is more direct and also allows us to handle the case of non-zero regularization ($\eps > 0$).

To further demonstrate the applicability of our technique, in Section~\ref{scn:commuting_matrices} we prove a generalization of Caffarelli's result: if $\nabla^2 V \preceq A^{-1}$ and $\nabla^2 W \succeq B^{-1}$, where $A$ and $B$ are arbitrary commuting positive definite matrices, then the Hessian of the Brenier potential from $P$ to $Q$ is pointwise upper bounded (in the PSD ordering) by $A^{-1/2} B^{1/2}$. This result implies a remarkable extremal property of optimal transport maps between Gaussian measures, namely: the optimal transport map from $\mc N(0, A)$ to $\mc N(0, B)$ maximizes the Hessian of the Brenier potential at any point among all possible measures $P$ and $Q$ satisfying our assumptions. To the best of our knowledge, this result is new.
\section{Background}\label{sec:background}

\subsection{Assumptions}

We study probability measures $P$ and $Q$ on $\R^d$ satisfying the following mild regularity assumptions.

\begin{assumption}[Regularity conditions]
We henceforth refer to the \textit{source measure} as $P$ and the \textit{target measure} as $Q$. We say that $(P, Q)$ satisfies our regularity conditions if:
\begin{enumerate}
    \item $P$ has full support on $\R^d$ and $Q$ is supported on a convex subset of $\R^d$. Let $\Omega_Q$ denote the interior of the support of $Q$, so that $\Omega_Q$ is a convex open set.
    \item $P$ and $Q$ admit positive Lebesgue densities on $\R^d$ and $\Omega_Q$, which we can therefore be written $\exp(-V)$ and $\exp(-W)$ respectively for functions $V, W : \R^d\to\R \cup \{\infty\}$. We abuse notation and identify the measures with their densities, thus writing $P = \exp(-V)$ and $Q = \exp(-W)$.
    \item We assume that $V$ and $W$ are twice continuously differentiable on $\R^d$ and $\Omega_Q$ respectively.
\end{enumerate}
\end{assumption}

Some of these assumptions can be eventually relaxed, but they suffice for the purposes of this work. Throughout the rest of the paper and for the sake of simplicity, these regularity assumptions are assumed to hold for the probability measures under consideration.

\subsection{Optimal transport without regularization}\label{sec: background_ot}

Let $P$ and $Q$ be probability measures with finite second moment.
The \emph{optimal transport problem} is the optimization problem
\begin{align}\label{eq: kant_p}
    \minimize_{\pi \in \Pi(P,Q)}~~~\int \tfrac{1}{2}\,\norm{x-y}^2 \, \D \pi(x,y)
\end{align}
where $\Pi(P,Q)$ is the set of joint probability measures with marginals $P$ and $Q$.
The following fundamental result characterizes the optimal solution to~\eqref{eq: kant_p}.

\begin{theorem}[{Brenier's theorem}]\label{thm: brenier_thm}
Suppose that $P$ admits a density with respect to Lebesgue measure.
Then, there exists a proper, convex, lower semicontinuous function $\phi_0 : \R^d\to\R \cup \{\infty\}$ such that the optimal transport plan in~\eqref{eq: kant_p} can be written $\pi_0 = {(\id, \nabla \phi_0)}_\sharp P$. The function $\phi_0$ is called the \emph{Brenier potential}, and the mapping $\nabla \phi_0$ is called the \emph{optimal transport map} from $P$ to $Q$.
Moreover, the optimal transport map $\nabla \phi_0$ is unique up to $P$-almost everywhere equality.

The Brenier potential $\phi_0$ is obtained as the solution to the dual problem
\begin{align}\label{eq: ot_dual}
\maximize_{\phi \in \Gamma_0}~~~\int \bigl(\frac{\norm\cdot^2}{2} -\phi\bigr) \dd P + \int \bigl(\frac{\norm\cdot^2}{2} -\phi^*\bigr) \dd Q\,,
\end{align}
where $\phi^*$ is the convex conjugate to $\phi$, and $\Gamma_0$ is the set of proper, convex, lower semicontinuous functions on $\R^d$.
\end{theorem}

We refer to~\cite{villani2003topics} for further background.
\subsection{Optimal transport with entropic regularization}\label{sec: background_eot}
\textit{Entropic optimal transport} is the problem that arises when we add the {Kullback--Liebler (KL) divergence}, denoted $D_\text{KL}(\cdot\mmid\cdot)$, as a regularizer to~\eqref{eq: kant_p}:
\begin{align}\label{eq: eot_p}
    \minimize_{\pi \in \Pi(P,Q)}~~~\int\tfrac{1}{2}\,\norm{x-y}^2 \dd \pi(x,y) + \eps \, D_{\rm KL}(\pi \mmid P\otimes Q)\,.
\end{align}
The following statement characterizes the solution to~\eqref{eq: eot_p}~\cite{Csi75, PeyCut19, nutz2021entropic}. 
\begin{theorem}[Entropic optimal transport]
Let $P$ and $Q$ be probability measures on $\R^d$ and fix $\eps > 0$. Then there exists a unique solution $\pi_\eps \in \Pi(P, Q)$ to~\eqref{eq: eot_p}.
Moreover, $\pi_\eps$ has the form
\begin{align}\label{eq:entropic_plan}
    \pi_\eps(\D x,\D y) = \exp\Bigl(\frac{f_\eps(x) + g_\eps(y) - \frac{1}{2} \, \norm{x-y}^2}{\eps}\Bigr) \, P(\D x) \, Q(\D y)\,,
\end{align}
where $(f_\eps,g_\eps)$ are maximizers for the dual problem
\begin{align}
    \maximize_{(f,g) \in L^1(P) \times L^1(Q)}~~~\int f \dd P + \int g \dd Q 
    &-\eps \iint e^{(f(x)+g(y)-\frac{1}{2} \, \norm{x-y}^2)/\eps}\dd P(x)\dd Q(y) + \eps\,.
\end{align}
\end{theorem}
The constraint that $\pi_\eps$ has marginals $P$ and $Q$ implies the following dual optimality conditions for $(f_\eps,g_\eps)$ (see \cite{mena2019statistical,nutz2021entropic}):
\begin{align}
    &f_\eps(x) = -\eps\log\int e^{(g_\eps(y) - \frac{1}{2} \, \norm{x-y}^2)/\eps}\dd Q(y) \qquad (x \in \R^d)\,, \label{eq: opt_cond1}\\
    &g_\eps(y) = -\eps\log\int e^{(f_\eps(x) - \frac{1}{2} \, \norm{x-y}^2)/\eps}\dd P(x) \qquad (y \in \R^d)\,. \label{eq: opt_cond2}
\end{align}
In particular, $f_\eps$ and $g_\eps$ are smooth. In this work, it is more convenient to work with the \textit{entropic Brenier potentials}, defined as
\begin{align}\label{eq: ent_brenier}
    (\phi_\eps,\psi_\eps) := (\tfrac12\,\|\cdot\|^2 - f_\eps,\;\tfrac12\,\|\cdot\|^2 - g_\eps)\,.
\end{align}
Since $(f_\varepsilon, g_\varepsilon)$ are only unique up to adding a constant to $f_\varepsilon$ and subtracting the same constant from $g_\varepsilon$, we fix the normalization convention $\int f_\varepsilon \, \D P = \int g_\varepsilon \, \D Q$. Under this condition, it was shown in~\cite{nutz2021entropic}
that we have convergence to the Brenier potential $\varphi_\eps \to \varphi_0$ as $\eps \searrow 0$.

Adopting this new notation, with $P = \exp(-V)$ and $Q = \exp(-W)$, we can rewrite the entropic optimal plan as
\begin{align*}
    \pi_\eps(\D x,\D y) = \exp\Bigl( - \frac{\varphi_\eps(x) + \psi_\eps(y) - \langle x, y\rangle }{\eps} - V(x) - W(y)\Bigr) \, \D x \, \D y\,.
\end{align*}

The entropic Brenier potentials were first introduced to develop a computationally tractable estimator of the optimal transport map $\nabla\phi_0$~\cite{seguy2017large,pooladian2021entropic,pooladian2022debiaser}.
Indeed, this is motivated by the following observation, which acts as an entropic version of Brenier's theorem.
Write $\pi_\eps^{Y\mid X=x}$ for the conditional distribution of $Y$ given $X=x$ for $(X,Y) \sim \pi_\eps$, and similarly define $\pi_\eps^{X\mid Y=y}$.
Then, by~\cite[Proposition 1]{pooladian2021entropic}, $\nabla \phi_\eps$ is the barycentric projection
\begin{equation}\label{eq: t_eps_expectation}
    \nabla \phi_\eps(x)
    = \int y \, \D \pi_\eps^{Y\mid X=x}(y)\,.
\end{equation}
For clarity of exposition, we abuse notation and abbreviate $\pi_\eps^{Y\mid X=x}$ by $\pi_\eps^x$ and $\pi_\eps^{X\mid Y=y}$ by $\pi_\eps^y$ when there is no danger of confusion.

The following lemma is a straightforward computation using~\eqref{eq:entropic_plan},~\eqref{eq: opt_cond1}, and~\eqref{eq: opt_cond2}.

\begin{lemma}\label{lem: hessians_pi}
It holds that
\begin{align*}
    & \nabla^2\phi_\eps(x) = \eps^{-1}\cov_{Y\sim\pi_\eps^x}(Y)\,, \qquad \text{ and } \qquad \nabla^2\psi_\eps(y) = \eps^{-1}\cov_{X\sim\pi_\eps^y}(X)\,.
\end{align*}
In particular, both $\varphi_\eps$ and $\psi_\eps$ are convex.
Moreover, under our regularity conditions,
\begin{align*}
    & \nabla_y^2\log(1/\pi_\eps^x)(y) = \eps^{-1}\, \nabla^2\psi_\eps(y) + \nabla^2W(y)\,,\\
    & \nabla_x^2\log(1/\pi_\eps^y)(x) = \eps^{-1}\, \nabla^2\phi_\eps(x) + \nabla^2V(x)\,.
\end{align*}
\end{lemma}

\subsection{Covariance inequalities}

In our proofs, we make use of the following key inequalities.

\begin{lemma}\label{lem:key}
Let $P = \exp(-V)$ be a probability measure on $\R^d$ and assume that $V$ is twice continuously differentiable on the interior of its domain.
Then, the following hold.
\begin{enumerate}
    \item (Brascamp--Lieb inequality) If in addition we assume that $P$ is strictly log-concave, then it holds that $\cov_{X\sim P}(X) \preceq \E_{X\sim P}[{(\nabla^2 V(X))}^{-1}]$.
    \item (Cram\'er--Rao inequality) $\cov_{X\sim P}(X) \succeq {(\E_{X\sim P}[\nabla^2 V(X)])}^{-1}$.
\end{enumerate}
\end{lemma}

The Brascamp--Lieb inequality is classical, and we refer readers to~\cite{bobkovledoux2000brunnmintobrascamplieblsi, bakrygentilledoux2014, cordero2017transport} for several proofs. To make our exposition more self-contained, we provide a proof of the Cram\'er--Rao inequality in the appendix.
\section{Main theorem}

We now state and prove our main theorem.

\begin{theorem}\label{thm:main}
    Let $P = \exp(-V)$ and $Q = \exp(-W)$.
    \begin{enumerate}
        \item Suppose that $(P, Q)$ satisfy our regularity assumptions, as well as
        \begin{align*}
            \nabla^2 V \preceq \beta_V I\,, \qquad\text{and}\qquad \nabla^2 W \succeq \alpha_W I \succ 0\,.
        \end{align*}
        Then, for every $\varepsilon > 0$ and all $x\in \R^d$, the Hessian of the entropic Brenier potential satisfies
        \begin{align*}
            \nabla^2 \varphi_\eps(x)
            \preceq \frac{1}{2} \, \bigl(\sqrt{4\beta_V/\alpha_W + \varepsilon^2 \beta_V^2} - \varepsilon \beta_V\bigr)\, I\,.
        \end{align*}
        \item Suppose that $(Q, P)$ satisfy our regularity assumptions, as well as
        \begin{align*}
            \nabla^2 V \succeq \alpha_V I \succ 0\,, \qquad\text{and}\qquad \nabla^2 W \preceq \beta_W I\,.
        \end{align*}
        Then, for every $\varepsilon > 0$ and all $x\in \Omega_P := \on{int}(\on{supp}(P))$, the Hessian of the entropic Brenier potential satisfies
    \begin{align*}
        \nabla^2 \varphi_\eps(x)
        &\succeq \frac{1}{2} \, \bigl(\sqrt{4\alpha_V/\beta_W + \varepsilon^2 \alpha_V^2} - \varepsilon \alpha_V\bigr) \, I\,.
        \end{align*}
    \end{enumerate}
\end{theorem}

Observe that as $\varepsilon \searrow 0$, we formally expect the following bounds on the Brenier potential:
\begin{align*}
    \sqrt{\alpha_V/\beta_W}\, I \preceq \nabla^2 \varphi_0(x) \preceq \sqrt{\beta_V/\alpha_W} \, I\,.
\end{align*}
In particular, this recovers Caffarelli's contraction theorem (Theorem~\ref{thm: caf_con}).
We make this intuition rigorous below by appealing to convergence results for the entropic potentials as the regularization parameter $\varepsilon$ tends to zero.

\begin{proof}[Proof of Theorem~\ref{thm:main}]
\textbf{Upper bound.}
Fix $x \in \R^d$. Recall from Lemma~\ref{lem: hessians_pi} that
\begin{align*}
    \nabla^2\varphi_\eps(x) = \eps^{-1} \cov_{Y \sim \pi_\eps^x}(Y)\,.
\end{align*}
By an application of the Brascamp{--}Lieb inequality, this results in the upper bound
\begin{align*}
    \nabla^2\varphi_\eps(x) &= \eps^{-1}\cov_{Y\sim \pi_\eps^x}(Y) \\
    &\preceq \eps^{-1}\E_{Y \sim \pi_\eps^x}\bigl[ \bigl(\eps^{-1}\,\nabla^2\psi_\eps(Y) + \nabla^2 W(Y)\bigr)^{-1} \bigr] \\
    &\preceq \E_{Y \sim \pi_\eps^x}\bigl[ \bigl(\nabla^2\psi_\eps(Y) + \eps \alpha_W I\bigr)^{-1}\bigr]\,,
\end{align*}
where in the last inequality we also used the lower bound on the spectrum of $\nabla^2W$.
Next, using Lemma~\ref{lem: hessians_pi} and the Cram\'er--Rao inequality (Lemma~\ref{lem:key}), we obtain the lower bound
\begin{align*}
    \nabla^2 \psi_\eps(Y)
    &= \eps^{-1}\cov_{X\sim \pi_\eps^Y}(X) \\
    &\succeq \eps^{-1} \, \bigl(\E_{X \sim \pi_\eps^Y}\bigl[ \eps^{-1}\,\nabla^2\varphi_\eps(X) + \nabla^2V(X) \bigr] \bigr)^{-1} \\
    &\succeq \bigl(\E_{X \sim \pi_\eps^Y}\bigl[ \nabla^2\varphi_\eps(X) + \eps \beta_V I  \bigr]\bigr)^{-1}\,,
\end{align*}
where we used the upper bound on the spectrum of $\nabla^2 V$. Combining these inequalities,
\begin{align*}
    \nabla^2 \varphi_\eps(x) &\preceq \E_{Y \sim \pi_\eps^x}\Bigl[ \Bigl( \bigl(\E_{X \sim \pi_\eps^Y}\bigl[ \nabla^2\varphi_\eps(X) + \eps\beta_V I  \bigr]\bigr)^{-1} + \eps\alpha_WI \Bigr)^{-1}\Bigr]\,.
\end{align*}

Now, define the quantity
\begin{align*}
    L_\eps := \sup_{x\in\R^d} \lambda_{\max}\bigl(\nabla^2\varphi_\varepsilon(x)\bigr)\,.
\end{align*}
Then, we have shown
\begin{align*}
    \lambda_{\max}\bigl(\nabla^2 \varphi_\varepsilon(x)\bigr)
    &\le \bigl((L_\eps + \varepsilon \beta_V)^{-1} + \varepsilon \alpha_W\bigr)^{-1}\,.
\end{align*}
Taking the supremum over $x \in \R^d$,
\begin{align*}
    L_\eps &\le \bigl((L_\eps + \varepsilon \beta_V)^{-1} + \varepsilon \alpha_W\bigr)^{-1}\,.
\end{align*}
Solving the inequality yields
\begin{align}\label{eq:Leps_bd}
    L_\eps \leq \frac{1}{2} \,\bigl(\sqrt{4\beta_V/\alpha_W + \eps^2\beta_V^2} - \eps\beta_V\bigr)\,.
\end{align}

\textbf{Lower bound.} The lower bound argument is symmetric, but we give the details for completeness. Using Lemma~\ref{lem: hessians_pi} and the Cram\'er--Rao inequality (Lemma~\ref{lem:key}),
\begin{align*}
    \nabla^2 \varphi_\varepsilon(x)
    &= \eps^{-1} \cov_{Y\sim \pi_\eps^x}(Y) \\
    &\succeq \eps^{-1} \, \bigl(\E_{Y\sim\pi_\eps^x}\bigl[ \eps^{-1} \, \nabla^2 \psi_\eps(Y) + \nabla^2 W(Y) \bigr]\bigr)^{-1} \\
    &\succeq \bigl(\E_{Y\sim\pi_\eps^x}\bigl[ \nabla^2 \psi_\eps(Y) + \eps \beta_W I\bigr]\bigr)^{-1}\,.
\end{align*}
Applying Lemma~\ref{lem: hessians_pi} and the Brascamp--Lieb inequality (Lemma~\ref{lem:key}),
\begin{align*}
    \nabla^2 \psi_\eps(Y)
    &= \eps^{-1} \cov_{X\sim \pi_\eps^Y}(X) \\
    &\preceq \eps^{-1} \E_{X\sim\pi_\eps^Y}\bigl[\bigl( \eps^{-1} \, \nabla^2 \varphi_\eps(X) + \nabla^2 V(X) \bigr)^{-1} \bigr] \\
    &\preceq \E_{X\sim\pi_\eps^Y}\bigl[\bigl( \nabla^2 \varphi_\eps(X) + \eps \alpha_V I \bigr)^{-1} \bigr]\,.
\end{align*}
Combining the two inequalities and setting
\begin{align*}
    \ell_\eps
    &:= \inf_{x\in\Omega_P} \lambda_{\min}\bigl(\nabla^2 \varphi_\eps(x)\bigr)\,,
\end{align*}
we deduce that
\begin{align*}
    \ell_\eps
    &\ge \bigl((\ell_\eps + \eps \alpha_V)^{-1} + \eps \beta_W\bigr)^{-1}\,.
\end{align*}
On the other hand, from Lemma~\ref{lem: hessians_pi}, we know that $\ell_\eps \ge 0$.
Solving the inequality then yields
\begin{align*}
    \ell_\eps
    &\ge \frac{1}{2} \, \bigl(\sqrt{4\alpha_V/\beta_W +\eps^2 \alpha_V^2} - \eps \alpha_V\bigr)\,. \qedhere
\end{align*}
\end{proof}

Next, we rigorously deduce Caffarelli's contraction theorem from Theorem~\ref{thm:main}.

\begin{proof}[Proof of Caffarelli's contraction (Theorem~\ref{thm: caf_con})]
    For every $\varepsilon > 0$, by Theorem~\ref{thm:main}, we have proven that $\nabla^2 \varphi_\varepsilon \preceq L_\varepsilon I$, with $L_\eps$ as in~\eqref{eq:Leps_bd}. Equivalently, this can be reformulated as saying that $\frac{L_\varepsilon \, \norm \cdot^2}{2} - \varphi_\varepsilon$ is convex.
    Fix some $\delta > 0$; in particular, for $\varepsilon$ sufficiently small, $\frac{(\sqrt{\beta_V/\alpha_W} + \delta) \, \norm \cdot^2}{2} - \varphi_\varepsilon$ is convex.
    
    Upon passing to a sequence $\varepsilon_k \searrow 0$, existing results on the convergence of entropic optimal transport potentials show that $\varphi_{\varepsilon_k} \to \varphi_0$ in $L^1(P)$ (see~\cite{nutz2021entropic}). Passing to a further subsequence, we obtain $\varphi_{\varepsilon_k} \to \varphi_0$ ($P$-almost surely). It follows that $\frac{(\sqrt{\beta_V/\alpha_W} + \delta) \, \norm \cdot^2}{2} - \varphi_0$ is convex for every $\delta > 0$ (see the remark after~\cite[Theorem 25.7]{rockafellar1997convexanalysis}), and thus for $\delta = 0$ as well.
\end{proof}

\begin{remark}
    Our main theorem provides both upper and lower bounds for $\nabla^2 \varphi_\eps$.
    In the case when $\eps = 0$, the lower bound follows from the upper bound.
    Indeed, if $\varphi_0$ is the Brenier potential for the optimal transport from $P$ to $Q$, then the convex conjugate $\varphi_0^*$ is the Brenier potential for the optimal transport from $Q$ to $P$. By applying Caffarelli's contraction theorem to $\varphi_0^*$ and appealing to convex duality, it yields a lower bound on $\nabla^2 \varphi_0$.
    However, we are not aware of a method of deducing the lower bound from the upper bound for positive values of $\eps$.
\end{remark}

\begin{remark}
    In Appendix~\ref{sec: gaussian_case}, by inspecting the Gaussian case, we show that Theorem~\ref{thm:main} is sharp for every $\eps > 0$.
\end{remark}

An inspection of the proof of the upper bound in Theorem~\ref{thm:main} reveals the following more general pair of inequalities.

\begin{proposition}\label{prop:main}
    Let $(P, Q)$ be probability measures satisfying our regularity conditions.
    Then, for all $x \in \R^d$ and $y \in \Omega_Q$,
    \begin{align*}
        \nabla^2 \varphi_\eps(x)
        &\preceq \E_{Y\sim \pi_\eps^x}\bigl[\bigl( \nabla^2 \psi_\eps(Y) + \eps \, \nabla^2 W(Y)\bigr)^{-1}\bigr]\,, \\
        \nabla^2 \psi_\eps(y)
        &\succeq \bigl(\E_{X\sim\pi_\eps^y}\bigl[ \nabla^2 \varphi_\eps(X) + \eps \, \nabla^2 V(X)\bigr] \bigr)^{-1}\,.
    \end{align*}
\end{proposition}

In the next section, we use these inequalities to prove a generalization of Caffarelli's  theorem.
\section{A generalization to commuting positive definite matrices}\label{scn:commuting_matrices}

In the next result, we replace the main assumptions of Caffarelli's contraction theorem, namely $\nabla^2 V \preceq \beta_V I$ and $\nabla^2 W \succeq \alpha_W I$, by the conditions
\begin{align}\label{eq:general_matrices}
    \nabla^2 V \preceq A^{-1} \qquad\text{and}\qquad \nabla^2 W \succeq B^{-1}\,,
\end{align}
where $A$ and $B$ are commuting positive definite matrices.
Recall that the Hessian of the Brenier potential between the Gaussian distributions $\mc N(0, A)$ and $\mc N(0,B)$ is the matrix $A^{-1/2} B^{1/2}$~\cite{gelbrich1990formula}. In light of this observation, the following theorem is sharp for every pair of commuting positive definite $(A,B)$, and shows that the Brenier potential between Gaussians achieves the largest possible Hessian among all source and target measures obeying the constraint~\eqref{eq:general_matrices}.

\begin{theorem}\label{thm:commuting_matrices}
    Let $(P,Q)$ satisfy our regularity conditions as well as the condition~\eqref{eq:general_matrices}. Then, the Hessian of the Brenier potential satisfies the uniform bound: for all $x\in \R^d$, it holds that
    \begin{align*}
        \nabla^2 \varphi_0(x)
        \preceq A^{-1/2} B^{1/2}\,.
    \end{align*}
\end{theorem}

As in Theorem~\ref{thm:main}, the proof technique also yields a lower bound on $\nabla^2 \varphi_0$ under appropriate assumptions. We omit this result because it is straightforward.

\begin{proof}
Let $C_\varepsilon$ be the smallest constant $C \ge 0$ such that $\nabla^2 \varphi_\eps(x) \preceq A^{-1/2} B^{1/2} + C I$ for all $x\in \R^d$. In light of Theorem~\ref{thm:main}, $C_\eps$ is well-defined and finite.
Equivalently,
\begin{align*}
    C_\eps
    &= \sup_{x\in\R^d} \sup_{e\in \R^d, \; \norm e = 1}{\bigl\langle e, [\nabla^2 \varphi_\eps(x) - A^{-1/2} B^{1/2}] \, e\bigr\rangle}\,.
\end{align*}
Let $(x, e)$ achieve the above supremum. (If the supremum is not attained, then the rest of the proof goes through with minor modifications.)

Using our assumptions and Proposition~\ref{prop:main}, we obtain
\begin{align*}
    C_\eps
    &= \bigl\langle e, [\nabla^2 \varphi_\eps(x) - A^{-1/2} B^{1/2}] \, e \bigr\rangle \\
    &\le \Bigl\langle e, \Bigl[ \bigl(( A^{-1/2} B^{1/2} + C_\eps I + \eps A^{-1})^{-1} + \eps B^{-1}\bigr)^{-1} - A^{-1/2} B^{1/2} \Bigr] \, e \Bigr\rangle\,.
\end{align*}
From our assumptions and Theorem~\ref{thm:main}, we know that the spectrum of $M_\eps := A^{-1/2} B^{1/2} + C_\eps I$ is bounded away from zero and infinity as $\eps \searrow 0$, which justifies the Taylor expansion
\begin{align*}
    \bigl(( M_\eps + \eps A^{-1})^{-1} + \eps B^{-1}\bigr)^{-1}
    &= \bigl( M_\eps^{-1} - \eps M_\eps^{-1} A^{-1} M_\eps^{-1} + \eps B^{-1} + O(\eps^2)\bigr)^{-1} \\
    &= M_\eps + \eps A^{-1} - \eps M_\eps B^{-1} M_\eps + O(\eps^2)\,.
\end{align*}
Hence,
\begin{align*}
    C_\eps
    &\le \bigl\langle e, \bigl[M_\eps + \eps A^{-1} - \eps M_\eps B^{-1} M_\eps + O(\eps^2) - A^{-1/2} B^{1/2}\bigr] \, e \bigr\rangle \\
    &\le C_\eps + \eps \, \bigl\langle e, [A^{-1} - M_\eps B^{-1} M_\eps] \, e \bigr\rangle + O(\eps^2) \\
    &= C_\eps + \eps \, \bigl\langle e, [C_\eps A^{-1/2} B^{-1/2} + C_\eps^2 B^{-1}] \, e \bigr\rangle + O(\eps^2)\,.
\end{align*}
This shows that $\lim_{\eps \searrow 0} C_\eps = 0$ (otherwise ${(C_\eps)}_{\eps > 0}$ would have a strictly positive cluster point which would contradict the above inequality for small enough $\eps > 0$).

By combining this fact with convergence of the entropic Brenier potentials as in the proof of Theorem~\ref{thm: caf_con}, we deduce the result.
\end{proof}

Next, we show how our theorem recovers and extends a result of Valdimarsson~\cite{valdimarsson2007otpotential}. Valdimarsson proves that if:
\begin{itemize}
    \item $\bar A$, $\bar B$, and $G$ are positive definite matrices;
    \item $\bar A \preceq G$ and $\bar B$ commutes with $G$;
    \item $P = \mc N(0, \bar B G^{-1}) * \mu$ where $*$ denotes convolution and $\mu$ is an arbitrary probability measure on $\R^d$; and
    \item $Q = \exp(-W)$ with $\nabla^2 W \succeq \bar B^{-1/2} \bar A^{-1}\bar B^{-1/2}$;
\end{itemize}
then the Brenier potential satisfies $\nabla^2 \varphi_0 \preceq G$. This result was then used to derive new forms of the Brascamp--Lieb inequality.\footnote{This is a different Brascamp--Lieb inequality than the one in Lemma~\ref{lem:key}.}

To prove this result, we first check that convolution with any probability measure only makes the density more log-smooth.

\begin{lemma}
    Let $\widetilde P \propto \exp(-\widetilde V)$ be a probability measure, where $\widetilde V : \R^d\to\R$ is twice continuously differentiable.
    Let $P := \widetilde P * \mu = \exp(-V)$ where $\mu$ is any probability measure on $\R^d$.
    Suppose that for some positive definite matrix $A^{-1}$, we have $\nabla^2 \widetilde V \preceq A^{-1}$.
    Then, $\nabla^2 V \preceq A^{-1}$ as well.
\end{lemma}
\begin{proof}
    An elementary computation shows that if we define the probability measure
    \begin{align*}
        \nu_y(\D x)
        &:= \frac{\exp(-\widetilde V(y-x)) \, \mu(\D x)}{\int \exp(-\widetilde V(y-x')) \, \mu(\D x')}
    \end{align*}
    then
    \begin{align*}
        \nabla^2 V(y)
        &= \E_{X \sim \nu_y}\bigl[\nabla^2 \widetilde V(y-X)\bigr] - \cov_{X\sim \nu_y}\bigl(\nabla \widetilde V(y-X)\bigr)\,,
    \end{align*}
    from which the result follows.
\end{proof}

From the lemma, we deduce that under Valdimarsson's assumptions, for $P = \exp(-V)$, we have $\nabla^2 V \preceq \bar B^{-1} G$.
Also, $\nabla^2 W \succeq \bar B^{-1/2} \bar A^{-1} \bar B^{-1/2} \succeq \bar B^{-1} G^{-1}$. By Theorem~\ref{thm:commuting_matrices}, the Brenier potential $\varphi_0$ satisfies $\nabla^2 \phi_0 \preceq G$. However, it is seen that our argument yields much more. For example, rather than requiring $P$ to be a convolution with a Gaussian measure, we can allow $P$ to be a convolution with any measure $\exp(-\widetilde V)$ satisfying $\nabla^2 \widetilde V \preceq \bar B^{-1} G$.

\begin{remark}
    It is natural to ask whether Theorem~\ref{thm:commuting_matrices} can be obtained by first applying Caffarelli's contraction theorem to show that the optimal transport map $\widetilde T_0$ between the measures $(A^{-1/2})_\sharp P$ and $(B^{-1/2})_\sharp Q$ is $1$-Lipschitz, and then considering the mapping $T_0(x) := B^{1/2} \widetilde T_0(A^{-1/2} x)$. Although $T_0$ is indeed a valid transport mapping from $P$ to $Q$, under our assumptions $\nabla T_0$ is not guaranteed to be symmetric, so it does not make sense to ask whether or not $\nabla T_0 \preceq A^{-1/2} B^{1/2}$.
    
    In Valdimarsson's application to Brascamp--Lieb inequalities, it is crucial that the transport map $T_0$ is chosen so that $\nabla T_0$ is a symmetric positive definite matrix. Symmetry of $\nabla T_0$ implies that $T_0$ is the gradient $\nabla \phi_0$ of a function $\phi_0 : \R^d\to\R$, and positive definiteness implies that $\phi_0$ is convex. By Brenier's theorem, the unique gradient of a convex function that pushes forward $P$ to $Q$ is the optimal transport map. Thus, it is crucial that we consider the \emph{optimal} transport map here; in particular, alternative maps such as the ones in~\cite{kim2012generalization, mikulincer2021brownian} cannot be applied.
\end{remark}
\section{Discussion}

We have proven a generalization of Caffarelli's celebrated theorem on the Lipschitz properties of the optimal transport map to the setting of entropic optimal transport using two complementary covariance inequalities (the Brascamp--Lieb inequality and the Cram\'er--Rao inequality).

We conjecture that our proof technique can also be used to recover the bounds on the moment measure mapping in~\cite{klartag2014momentmeasure}, provided that the existence of an ``entropic moment measure'' can be established (with convergence towards the true moment measure as the regularization tends to zero). As this is outside the scope of this work, we do not pursue this question here.

\paragraph*{Acknowledgements.}
We thank Ramon van Handel, Jonathan Niles-Weed, and Philippe Rigollet for helpful comments. This work was completed while SC was visiting NYU\@. SC was supported by
the Department of Defense (DoD) through the National Defense Science \& Engineering Graduate Fellowship (NDSEG) Program. AAP was partially supported by the Natural Sciences and Engineering Research Council of
Canada.

\bibliography{workingtitle}
\bibliographystyle{alpha}

\newpage
\appendix
\section{Proof of the Cram{\'e}r--Rao lower bound}\label{sec:cr_proof}

\begin{proof}[Proof of Lemma~\ref{lem:key}, Cram\'er--Rao inequality]
For any smooth and compactly supported test function $h : \R^d\to\R$, integration by parts yields
\begin{align*}
    \E_P \nabla h
    &= \int \nabla h \, \D P
    = -\int (h \, \nabla \ln P) \, \D P
    = \int (h - \E_P h) \, \nabla V \, \D P
\end{align*}
where we used the fact that $\E_P \nabla \ln P = 0$.
Therefore,
\begin{align}\label{eq:cr1}
    \langle \E_P \nabla h, (\E_P \nabla^2 V)^{-1} \, \E_P \nabla h\rangle
    &= \int (h - \E_P h) \, \langle \nabla V, (\E_P \nabla^2 V)^{-1} \, \E_P \nabla h \rangle \, \D P\,.
\end{align}
Applying the Cauchy--Schwarz inequality,
\begin{align*}
    \eqref{eq:cr1}
    &\le \sqrt{(\var_P h) \int \langle \E_P \nabla h, {(\E_P \nabla^2 V)}^{-1} \, {(\nabla V)}^{\otimes 2} \, {(\E_P \nabla^2 V)}^{-1}\, \E_P \nabla h \rangle \, \D P}\,.
\end{align*}
Integration by parts shows that $\int \nabla V^{\otimes 2} \, \D P = \int \nabla^2 V \, \D P$, and upon rearranging we deduce that
\begin{align}\label{eq:cr2}
    \var_P h
    &\ge \langle \E_P \nabla h, (\E_P \nabla^2 V)^{-1} \, \E_P \nabla h\rangle\,.
\end{align}
By approximation, this continues to hold for any locally Lipschitz $h : \R^d\to\R$ with $\E_P\norm{\nabla h} < \infty$.

Specializing the inequality~\eqref{eq:cr2} to $h := \langle e, \cdot\rangle$ for a unit vector $e\in\R^d$ then recovers the Cram\'er--Rao inequality of Lemma~\ref{lem:key}.
\end{proof}

\section{Gaussian case}\label{sec: gaussian_case}

Suppose $P = \mathcal{N}(0,A)$ and $Q = \mathcal{N}(0,B)$ are Gaussians.
Then, it is known that the Hessian of the Brenier potential is given by \cite{gelbrich1990formula}
\begin{align*}
    \nabla^2\varphi_0(x) = A^{-1/2} \, {(A^{1/2}BA^{1/2})}^{1/2} \,A^{-1/2}\,.
\end{align*}
If we have
\begin{align*}
    A^{-1} \preceq \beta I \qquad\text{and}\qquad B^{-1} \succeq \alpha I \succ 0\,,
\end{align*}
then Caffarelli's contraction theorem (Theorem~\ref{thm: caf_con}) implies
\begin{align*}
    \|\nabla^2\phi_0\|_{\text{op}} \leq \sqrt{\beta/\alpha}\,.
\end{align*}
This matches the bound of~\cite[Lemma 2]{altschuler2021averaging}.

For $\eps > 0$, the upper bound from Theorem~\ref{thm:main} implies
\begin{align}\label{eq:gaussian_eps}
    \norm{\nabla^2 \phi_\eps}_{\rm op} \leq \frac{1}{2} \, \bigl( \sqrt{4\beta/\alpha + \eps^2 \beta^2} - \eps \beta \bigr)\,.
\end{align}
On the other hand, from~\cite{janati2020entropic, mallasto2021entropy}, it is known that
\begin{align*}
    \nabla^2 \phi_\eps(x)
    &= A^{-1/2} \, \bigl( A^{1/2} B A^{1/2} + \frac{\eps^2}{4} \, I\bigr)^{1/2} \, A^{-1/2} - \frac{\eps}{2} \, A^{-1}\,.
\end{align*}
In particular, if we take $A = \beta^{-1} I$ and $B = \alpha^{-1} I$, then~\eqref{eq:gaussian_eps} is an equality. Hence, Theorem~\ref{thm:main} is sharp for every $\eps > 0$.
\end{document}